\documentclass[a4paper,12pt]{amsart}
\usepackage{amsmath,amssymb,amsfonts,latexsym,amsthm,fullpage}

\newcommand\Vbullet{\raisebox{-2.1pt}{\kern-0.4pt\vbox{\baselineskip4pt\lineskiplimit0pt%
\hbox{$\bullet$}\hbox{$\bullet$}}}}

\title{The hypermultiple gamma functions of BM-type}
\author{Hanamichi Kawamura}
\begin{document}
\maketitle
\newtheorem{Thm}{Theorem}[section]
\newtheorem{Def}[Thm]{Definition}
\newtheorem{Lem}[Thm]{Lemma}
\newtheorem{Prop}[Thm]{Proposition}
\newtheorem{Cor}[Thm]{Corollary}
\newtheorem{Ex}[Thm]{Example}
\newcommand{\bs}[1]{\boldsymbol{#1}}
\newcommand{\env}[1]{\begin{eqnarray*}\displaystyle#1\end{eqnarray*}}
\newcommand{\nmb}[1]{\begin{eqnarray}\displaystyle#1\end{eqnarray}}
\newcommand{\kak}[1]{\left(#1\right)}
\begin{abstract}
In this paper, we introduce the hypermultiple gamma functions of BM-type and prove the asymptotic expansion of these functions.
\end{abstract}
\section{Introduction}
In Barnes \cite{barnes}, he introduced the multiple gamma function
\env{\Gamma_r(w;\bs{\omega})=\exp\kak{\left.\frac{\partial}{\partial s}\zeta_r(s,w;\bs{\omega})\right|_{s=0}},}
where $\zeta_r$ means the Barnes multiple zeta function
\begin{eqnarray}\label{zetadef}\displaystyle \zeta_r(s,w;\bs{\omega})=\sum_{\bs{n}\geq\bs{0}} (\bs{n}\cdot\bs{\omega}+w)^{-s},\end{eqnarray}
$\mathrm{Re}(w),\mathrm{Re}(\omega_i)>0$ $(i=1,\cdots,r)$, $\bs{\omega}=(\omega_1,\cdots,\omega_r)$, $\bs{n}=(n_1,\cdots,n_r)\in\mathbb{Z}^r$, $\bs{n}\cdot\bs{\omega}=n_1\omega_1+\cdots+n\omega_r$ and $\bs{n}\geq\bs{0}$ means $n_i\geq{0}$ $(i=1,\cdots,r)$. He also proved the asymptotic expansion of this function:
\begin{Thm}
	As $w\to\infty$, we have asymptotically
	\env{\log\Gamma_r(w+a;\bs{\omega})=\sum_{N=0}^r a_{r,-N}(a;\bs{\omega})\frac{(-w)^{N}}{N!}(H_{N}-\log w)+a_{r,k+1}(a;\bs{\omega})\frac{1}{w}+O\left(\frac{1}{w^2}\right),}
	where $\mathrm{Re}(a)>0$, $|\bs{\omega}|=\omega_1\cdots\omega_r$, $H_N$ means $N$-th harmonic number
	\env{H_N=\sum_{i=1}^N i^{-1}}
	and $a_{r,N}$ means the $N$-th multiple Bernoulli polynomial
	\env{\frac{e^{-wt}}{\prod_{i=1}^r (1-e^{-\omega_it})}=\sum_{N=-r}^{\infty} a_{r,N}(w;\bs{\omega})t^N.}
	\end{Thm}
Katayama-Ohtsuki \cite{katayama1998} gave a simple proof of this behavior based on the integral representation of $\log \Gamma_r$;
\env{\log\Gamma_r(w;\bs{\omega})=\int_{I(\lambda,\infty)} f_{\bs{\omega}}(t)e^{-wt}t^{-1}\kak{\frac{1}{2\pi i}\log t+\kak{\frac{\gamma}{2\pi i}-\frac{1}{2}}}\,dt,}
where $f_{\bs{\omega}}(t)=\prod_{i=1}^r (1-e^{-\omega_it})^{-1}$, $0<\lambda<\mathrm{min}_{1\leq{i}\leq{r}} |2\pi/\omega_i|$ and $I(\lambda, \infty)$ is the path consisting of the infinite line from $\infty$ to $\lambda$, the circle of radius $\lambda$ around $0$ in the positive sense and the infinite line from $\lambda$ to $\infty$.\\

While, Kurokawa-Ochiai \cite{ki} introduced the multiple gamma functions of BM (Barnes-Milnor) type:
\env{\Gamma_{r,k}(w;\bs{\omega})=\exp\kak{\left.\frac{\partial}{\partial s}\zeta_r(s,w;\bs{\omega})\right|_{s=-k}}}
for non-negative integers $k$. They also generalized Kinkelin's formulas:
\begin{Thm}
For $k\geq{1}$, we have
\env{\int_0^w \kak{\log\Gamma_{r,k-1}(t;\bs{\omega})-\frac{1}{k}\zeta_r(1-k,t;\bs{\omega})}\,dt=\frac{1}{k}\log\frac{\Gamma_{r,k}(w;\bs{\omega})}{\Gamma_{r,k}(0;\bs{\omega})}.}
\end{Thm}
This theorem can be written as
\env{\frac{\partial}{\partial w}\log\Gamma_{r,k}(w;\bs{\omega})=k\log\Gamma_{r,k-1}(w;\bs{\omega})-\zeta_r(1-k,w;\bs{\omega}).}
The term $\zeta_r(1-k,w;\bs{\omega})$ can be regarded as a ``gap" of hierarchy of $\Gamma_{r,k}$. To correct this gap, the ``balanced" multiple gamma funtions $P_{r,k}$ were defined in the author's preprint \cite{ore1}:
\env{P_{r,k}(w;\bs{\omega})=\frac{(-1)^k}{k!}\log\Gamma_{r,k}(w;\bs{\omega})+H_ka_{r,k}(w;\bs{\omega})}
for $k\geq{0}$. Then hierarchy of $P_{r,k}$ is simpler than that of $\Gamma_{r,k}$:
\env{\frac{\partial}{\partial w}P_{r,k}(w;\bs{\omega})=-P_{r,k-1}(w;\bs{\omega}).}
This fact shows that $P_{r,k}$ can be defined for any $k\in\mathbb{Z}$. Moreover, we can investigate the asymptotic behavior of $P_{r,k}$ in the same way as we prove that of $\Gamma_r$:
\begin{Thm}
	As $w\to\infty$, we have asymptotically
	\env{P_{r+l,k}(w;(\bs{\omega},\bs{\alpha}))=\sum_{N=-l}^{r+k} a_{l,N}(a;\bs{\alpha})P_{r,k-N}(w;\bs{\omega})+\frac{a_{l,r+k+1}(a;\bs{\alpha})}{|\bs{\omega}|_{\times}}\frac{1}{w}+O\kak{\frac{1}{w^2}},}
	where $\bs{\alpha}=(\alpha_1,\cdots,\alpha_l)$ with $\mathrm{Re}(\alpha_j)>0$ $(j=1,\cdots,l)$.
\end{Thm}
In this paper, we introduce a new generalization of the multiple gamma functions and call it the hypermultiple gamma functions of BM-type:
\env{{}_m\Gamma_{r,k}(w;\bs{\omega})=\exp\kak{\left.\frac{\partial^m}{\partial s^m}\zeta_r(s,w;\bs{\omega})\right|_{s=-k}}}
This generalization includes both Kurokawa-Ochiai's $\Gamma_{r,k}$ and Katayama's hypermultiple gamma functions
\env{{}_m\Gamma_r(w;\bs{\omega})=\exp\kak{\left.\frac{\partial^m}{\partial s^m}\zeta_r(s,w;\bs{\omega})\right|_{s=0}}}
introduced in Katayama \cite{higher}. We remark that $\log{}_m\Gamma_r$ means the natural logarithm of ${}_m\Gamma_r$, not $\log\Gamma_r/\log m$. The name "hypermultiple" is derived from Katayama's report \cite{tsudajuku}.\\

Our second purpose is to construct the "balanced" hypermultiple gamma functions of BM-type ${}_mP_{r,k}$ and to show the asymptotic behavior of these functions. We shall give the definition of ${}_mP_{r,k}$ in the section 2. Our main theorem is following:
\begin{Thm}\label{mt}
	As $w\to\infty$, we have asymptotically
	\env{{}_mP_{r+l,k}(w;(\bs{\omega},\bs{\alpha}))=\sum_{N=-l}^{r+k} a_{l,N}(a;\bs{\alpha}){}_mP_{r,k-N}(w;\bs{\omega})+O\kak{\frac{(\log w)^{m-1}}{w}}.}
\end{Thm}
\section{Construction of ${}_mP_{r,k}$}
In this section, we construct the balanced hypermultiple gamma functions of BM-type ${}_mP_{r,k}$.
For $k\in\mathbb{Z}_{\geq{0}}$, we define
\env{\frac{e^{(s+k)x}}{\Gamma(s)(e^{2\pi is}-1)}=\sum_{m=0}^{\infty} \frac{{}_mQ_k(x)}{m!}(s+k)^m.}
Then ${}_mQ_k(x)$ is a polynomial of $x$ whose degree is $m$. From a well-known representation
\env{\zeta_r(s,w;\bs{\omega})=\frac{1}{\Gamma(s)(e^{2\pi is}-1)}\int_{I(\lambda,\infty)} f_{\bs{\omega}}(t)e^{-wt}t^{s-1}\,dt,}
we obtain a representation of the hypermultiple gamma functions of BM-type
\env{\log{}_m\Gamma_{r,k}(w;\bs{\omega})=\int_{I(\lambda,\infty)} f_{\bs{\omega}}(t)e^{-wt}t^{-k-1}{}_mQ_k(\log t)\,dt.}
For $\mu\in\mathbb{Z}_{\geq{0}}$, we define the multiple harmonic sum
\env{H_k(\mu)=\sum_{0<j_1\leq\cdots\leq j_{\mu}\leq k} j_1^{-1}\cdots j_{\mu}^{-1}}
and put
\env{c^m_{\mu,k}=\frac{(-1)^k}{k!}\frac{m!}{(m-\mu)!}H_k(\mu).}
We consider $H_0(\mu)$ as $0$ and $H_k(0)$ as $1$. Then we can define the balanced hypermultiple gamma functions of BM-type
\env{{}_mP_{r,k}(w;\bs{\omega})=\sum_{\mu=0}^m c^m_{m-\mu,k}\log{}_m\Gamma_{r,k}(w;\bs{\omega}).}
\begin{Prop}
	The functions ${}_mP_{r,k}$ satisfy simplified Kinkelin's formula. In other words, we have
	\env{\frac{\partial}{\partial w}{}_mP_{r,k}(w;\bs{\omega})=-{}_mP_{r,k-1}(w;\bs{\omega}).}
\end{Prop}
\begin{proof}
	From (\ref{zetadef}), we have
	\env{\frac{\partial}{\partial w}\zeta_r(s,w;\bs{\omega})=-s\zeta_r(s+1,w;\bs{\omega}).}
	Thus we get
	\env{\frac{\partial}{\partial w}\log{}_m\Gamma_{r,k}(w;\bs{\omega})=k\log{}_m\Gamma_{r,k-1}(w;\bs{\omega})-m\log{}_{m-1}\Gamma_{r,k-1}(w;\bs{\omega})}
	for $m\geq{1}$ and
	\env{\frac{\partial}{\partial w}\log{}_0\Gamma_{r,k}(w;\bs{\omega})=k\log{}_0\Gamma_{r,k}(w;\bs{\omega})}
	for $m=0$. Then we obtain
	\env{\frac{\partial}{\partial w}{}_mP_{r,k}(w;\bs{\omega})=\sum_{\mu=0}^{m-1} (kc^m_{m-\mu,k}-(\mu+1)c^m_{m-\mu-1,k})\log{}_m\Gamma_{r,k-1}(w;\bs{\omega})+kc^m_{0,k}\log{}_m\Gamma_{r,k-1}(w;\bs{\omega})}
	and
	\env{{}_mP_{r,k-1}(w;\bs{\omega})=\sum_{\mu=0}^{m-1} c^m_{m-\mu,k-1}\log{}_{\mu}\Gamma_{r,k-1}(w;\bs{\omega})+c^m_{0,k-1}\log{}_m\Gamma_{r,k-1}(w;\bs{\omega}).}
	Therefore we only have to show
	\env{kc^m_{\mu,k}-(m-\mu+1)c^m_{\mu-1,k}+c^m_{\mu,k-1}=0}
	for $\mu\geq{1}$ and
	\env{kc^m_{0,k}+c^m_{0,k-1}=0.}
	These identity are easliy derived from a relation
	\begin{eqnarray}\label{relofhkm}\displaystyle kH_k(\mu)-H_k(\mu-1)-kH_{k-1}(\mu)=0.\end{eqnarray} 
\end{proof}
	If we put
\env{S_{m,k}(x)=\sum_{\mu=0}^m {}_mQ_k(x)c^m_{m-\mu,k},}
we can obviously write the definition of ${}_mP_{r,k}$ as
\env{\log{}_mP_{r,k}(w;\bs{\omega})=\int_{I(\lambda,\infty)} f_{\bs{\omega}}(t)e^{-wt}t^{-k-1}S_{m,k}(\log t)\,dt.}
\section{Proof of the main theorem}
In this section, we prove some important fact about $S_{m,k}(x)$ and our main theorem. The following proposition plays an important role in proof of the main theorem. 
\begin{Prop}\label{unko}
	The polynomial $S_{m,k}(x)$ is independent on $k$.
\end{Prop}
\begin{proof}
	Let $F_k(s)$ be the generating function of $H_k(\mu)$ for fixed $k$:
	\env{F_k(s)=\sum_{\mu=0}^{\infty} H_k(\mu)s^{\mu}.}
	From an easy identity
	\begin{eqnarray}\label{id}\displaystyle \sum_{m=0}^{\infty} \frac{S_{m,k}(x)}{m!}s^m=\frac{(-1)^k}{k!}F_k(s)\frac{e^{sx}}{\Gamma(s-k)(e^{2\pi is}-1)},\end{eqnarray}
	we only have to show that the right side of (\ref{id}) is independent on $k$. Since (\ref{relofhkm}), it follows that
	\env{F_k(s)-F_{k-1}(s)=\frac{s}{k}F_k(s).}
	This relation and an identity $F_0(s)=1$ shows that
	\env{F_k(s)=\frac{k!}{(1-s)_k},}
	where $(a)_k=a(a+1)\cdots(a+k-1)$ is the Pochhammer symbol. Hence we have
	\begin{eqnarray}\label{0q0}\displaystyle \frac{(-1)^k}{k!}F_k(s)\frac{e^{sx}}{\Gamma(s-k)(e^{2\pi is}-1)}&=&\frac{e^{sx}}{\Gamma(s)(e^{2\pi is}-1)}.\end{eqnarray}
	This is clearly independent on $k$.
\end{proof}
\begin{Cor}
	For $m\geq{0}$, we have $S_{m,k}(x)={}_mQ_0(x)$.
\end{Cor}
\begin{proof}
	The right side of (\ref{0q0}) is equal to
	\env{\sum_{m=0}^{\infty} \frac{{}_mQ_0(x)}{m!}s^m.}
\end{proof}
\begin{proof}[Proof of Theorem \ref{mt}]
	From Theorem \ref{unko}, we obtain
	\begin{eqnarray}\label{ponzu}\displaystyle {}_mP_{r+l,k}(w+a;(\bs{\omega},\bs{\alpha}))&=&\int_{I(\lambda,\infty)} f_{\bs{\omega}}(t)e^{-wt}\kak{\sum_{N=r+k+1}^{\infty} a_{l,N}(a;\bs{\alpha})t^N}t^{-k-1}S_{m,k}(\log t)\,dt\nonumber\\&{}&+\sum_{N=-l}^{r+k}a_{l,N}(a;\bs{\alpha}){}_mP_{r,k-N}(w;\bs{\omega}).\end{eqnarray}
	From an identity
	\env{&&\int_{I(\lambda,\infty)}f_{\bs{\omega}}(t)e^{-wt}\kak{\sum_{N=r+k+1}^{\infty} a_{l,N}(a;\bs{\alpha})t^N}t^{-k-1}(\log t)^{\nu}\,dt\\&=&\sum_{D=0}^{\nu-1} \binom{\nu}{D}(2\pi i)^{\nu-D}\int_0^{\infty}f_{\bs{\omega}}(t)e^{-wt}\kak{\sum_{N=r+k+1}^{\infty} a_{l,N}(a;\bs{\alpha})t^N}t^{-k-1}(\log t)^D\,dt,}
	the first term of the right side of (\ref{ponzu}) is $O((\log w)^{m-1}/w)$.
\end{proof}

\end{document}